\documentclass[12pt,oneside,reqno]{amsart}

\usepackage{amsmath,amsrefs}
\usepackage{amsfonts}
\usepackage{amssymb}
\usepackage{amsthm}
\usepackage{bbm}
\usepackage{color, comment}
\usepackage{enumitem}
\usepackage[hidelinks]{hyperref}
\usepackage{tikz}
\usepackage[a4paper,textwidth=14cm]{geometry}
\usepackage[nameinlink]{cleveref}

\hypersetup{
    colorlinks=true,
    linkcolor=magenta,
    filecolor=magenta,      
    urlcolor=cyan,
}

\newcommand{\eps}{\varepsilon}
\renewcommand{\H}{\mathcal{H}}
\newcommand{\N}{\mathbb{N}}
\newcommand{\R}{\mathbb{R}}
\newcommand*{\genbf}[1]{\ifmmode\mathbf{#1}\else\textbf{#1}\fi}

\newcommand{\cH}{\mathcal H}
\newcommand{\G}{\mathbb G}

\newcommand{\cC}{\mathcal C}

\newcommand{\ch}{\mathfrak h}

\newcommand{\bH}{\mathbb H}
\newcommand{\sm}{\setminus}

\newcommand{\cG}{\mathcal G}

\newcommand{\beq}{\begin{equation}}
\newcommand{\eeq}{\end{equation}}

\newtheorem{theorem}{Theorem}[section]
\newtheorem{lemma}[theorem]{Lemma}
\newtheorem{proposition}[theorem]{Proposition}

\theoremstyle{definition}
\newtheorem{definition}[theorem]{Definition}
\newtheorem{remark}[theorem]{Remark}

\title{$C^{1,\alpha}$-rectifiability in low codimension in Heisenberg groups}

\author{Kennedy Obinna Idu}
\address{University of Toronto, Department of Mathematics, Toronto, ON M5S 2E4, Canada.}
\email{\href{mailto:idu@math.toronto.edu}{idu@math.toronto.edu}}
\author{Francesco Paolo Maiale}
\address{Gran Sasso Science Institute, Viale F. Crispi 7, 67100 L'Aquila, Italy.}
\email{\href{mailto:francescopaolo.maiale@gssi.it}{francescopaolo.maiale@gssi.it}}
\urladdr{\url{https://sites.google.com/site/francescopaolomaiale/}}

\date{}

\begin{document}
\begingroup
\def\uppercasenonmath#1{} 
\let\MakeUppercase\relax 
\maketitle
\endgroup


\begin{abstract}
A natural higher-order notion of $C^{1,\alpha}$-rectifiability, $0 < \alpha \leq 1$, is introduced for subsets of the Heisenberg groups $\bH^n$ in terms of covering a set almost everywhere with a countable union of $(\mathbf{C}_H^{1,\alpha},\bH)$-regular surfaces. Using this, we prove a geometric characterization of $C^{1,\alpha}$-rectifiable sets of low-codimension in Heisenberg groups $\bH^n$ in terms of an almost everywhere existence of suitable approximate tangent paraboloids. 
\end{abstract}

\vspace{1cm}
 {\footnotesize
 \textbf{MSC (2010):} 28A75 (primary); 43A80, 53C17 (secondary).

 \textbf{Keywords:} Higher order, approximate tangent paraboloids, metric spaces}.

\section{Introduction}

Rectifiable sets are focal to studies in geometric measure theory and admit various applications in several branches of mathematical analysis. Interest in such sets arises mainly for their geometric, measure-theoretic, and analytic properties, which include a notion of (approximate) tangent spaces defined almost everywhere, a version of the area and coarea formulas (see \cite{AmbKir2000Rect} and \cite{kirchheim1994rectifiable}), and a framework for studying boundedness of a class of singular integral operators (see, e.g., \cite{chousionis2019boundednes,DS91,DS93}).

In metric spaces, particularly Carnot groups, the definition of rectifiability diverges along several, not necessarily equivalent, directions (see, e.g., \cite{FSSC4, pauls2004notion, AntLeD19, FranchiSerapioni2016IntrLip}). The original definition by Federer \cite[Section 3.2.14]{Federer69} is in terms of composing a set with countably many Lipschitz images of subsets of the Euclidean space $\R^n$. This is adopted in \cite{AmbKir2000Rect} and shown to be inappropriate in general metric spaces considering even the basic setting of the Heisenberg group $\bH^1$. In \cite{MSSC10}, Mattila et al. defined rectifiability in the Heisenberg group $\bH^n$ considering a countable union of $C^1$ $\bH$-regular surfaces. This is related to using notions of regular surfaces in the sense of Franchi, Serapioni, and Serra Cassano (see, e.g., \cite{FSSC01,FSSC6,FSSC11}). Several results can be found on characterizations and basic properties of rectifiable sets in Euclidean and general metric spaces (see, e.g., \cite{AmbKir2000Rect, Federer69, Federer1978Colloquium, DeLellis08, Mattila, MSSC10}). A well-known characterization in the Heisenberg group $\bH^n$ is the a.e. existence of approximate tangent spaces \cite{MSSC10}. This is in the spirit of the Euclidean analog, which is in terms of an almost everywhere existence of approximate tangent planes (see, e.g., \cite[Corollary 15.16]{Mattila}).

A missing piece in the study of rectifiability in metric spaces is the natural notion of higher-order rectifiability, which can be defined in terms of composing a set with countably many objects of higher order regularity defined appropriately. Motivated by the seminal work of Anzellotti and Serapioni \cite{AnzSer} in the Euclidean setting, and also the recent study of Del~Nin and the first named author \cite{ninidu1}, our goal in this article is to initiate progress along this line in the metric setting of Heisenberg groups. We introduce a notion of $C^{1,\alpha}$-rectifiability, $0 < \alpha \leq 1$, defined in terms of composing a set with countably many $(C^{1,\alpha}_H,\bH)$-regular surfaces (see \Cref{Def:RegSurf} and \Cref{def:rec}). 

Using this, we address the problem of characterization of $C^{1,\alpha}$-rectifiable sets in a metric setting. An interesting, and perhaps gratifying, discovery is proving that the analogous geometric criterion of {\itshape approximate tangent paraboloids} as in the Euclidean characterization of $C^{1,\alpha}$-rectifiable sets, $0<\alpha\le 1$ (see \cite{ninidu1, San19}) applies in the setting of low-codimensional sets of the Heisenberg groups $\bH^n$.

Throughout the paper we write $k$ and $k_m$ as the {\itshape dimension} and {\itshape metric dimension} respectively, $\cG(\bH^n,k)$ is the Grassmannian of $k$-dimensional subgroups (see \Cref{def.grass}) and denote by $Q_\alpha(p,V,\lambda)$ the $\alpha$-paraboloid centered at $p$ with base $V$ and dilation $\lambda$ (see \Cref{def.objects}). We now state the main result of this paper:

\begin{theorem} \label{thm.1.1} 
Fix $\alpha \in (0,1]$ and $n<k\le 2n$. Let $E \subset \bH^n$ be a $\H^{k_m}$-measurable set with $\H^{k_m}(E) < \infty$ such that for $\H^{k_m}$-a.e. $p \in E$ there are $V_p \in \cG(\bH^n,k)$ and $\lambda > 0$ such that
\beq\label{eq.paraboloidassumption}
\lim_{r \to 0^+} \frac{1}{r^{k_m}} \H^{k_m} \left( E \cap B(p,r) \setminus Q_\alpha(p,V_p,\lambda) \right) = 0.
\eeq
Assume, in addition, that for $\H^{k_m}$-a.e. $p \in E$ there holds
\beq\label{eq.densityposit}
\Theta_\ast^{k_m}(E,p) > 0.
\eeq
Then $E$ is $C^{1,\alpha}$-rectifiable in the sense of \Cref{def:rec}.
\end{theorem}

Next, we prove that the converse also holds. As above, we are in the low-codimensional setting so $n < k \le 2n$.

\begin{proposition} \label{prop.1.2}
If $E \subset \bH^n$ is a $C^{1,\alpha}$-rectifiable set with $\H^{k_m}(E) < \infty$, then for $\H^{k_m}$-a.e. $p \in E$ there exist $V_p \in \cG(\bH^n,k)$ and $\lambda=\lambda_p > 0$ such that
\[
\lim_{r \to 0^+} \frac{1}{r^{k_m}} \H^{k_m} \left( E \cap B(p,r) \setminus Q_\alpha(p,V_p,\lambda) \right) = 0.
\]
\end{proposition}

\begin{remark}
In the Euclidean setting \cite[Theorem 1.1]{ninidu1} the positive lower density condition \eqref{eq.densityposit} is recovered from the approximate tangent paraboloid condition; however, in our setting, we have to impose the explicit requirement.

A similar issue arises in the characterization of $k$-rectifiable sets of low codimension in $\bH^n$ (see \cite[Theorem 3.15]{MSSC10}), where the positive lower density condition is required and it is asked if this can be removed. This, to the best of our knowledge, remains open.
\end{remark}

The proof of \Cref{thm.1.1} draws ideas from \cite[Lemma 3.5]{ninidu1} by first recovering the H\"{o}lder regularity of the distribution of subgroups (\Cref{lemma:keyHn}) using the density conditions. This is adeptly done with technical modifications of some analogous Euclidean lemmas and arguments which are necessary because of the more delicate structure of the Heisenberg groups. After proving the existence of a H\"{o}lder continuous selection of horizontal vector fields corresponding to the horizontal complements of the distribution of vertical subgroups we further adapt similar technique in \cite{MSSC10} and use a standard decomposition argument (based on the density conditions). 

It is interesting to notice that most of the technical points exploited to prove \Cref{thm.1.1} can be easily extended to a general Carnot group. Indeed, the structure of the Heisenberg group $\bH^n$ only plays a fundamental role in \Cref{prop.str} and, consequently, in the H\"older-continuity result \Cref{lemma.2.32.holder}. We will discuss this possible extension in more details in \Cref{sec:carnot}.

\subsection{Comparison with Euclidean analogue}

\Cref{thm.1.1}, to the best of our knowledge, is the first result in the direction of higher-order rectifiability in metric spaces. The result is inspired by the Euclidean analog \cite[Theorem 1.1]{ninidu1} in a joint work by Del~Nin and the first named author and the question of extension to non-Euclidean spaces. In the Heisenberg groups $\bH^n$, the idea comes to first understand the intricacies of the space and to consider alternative routes to deploying the strategy in the Euclidean setting. This requires a shrewd approach and numerous detours due to the pathological geometry of the Heisenberg groups, which includes nonlinearity of vertical projections, dimensional distortion under projections, and the lack of a Whitney extension result for mappings into $\bH^n$. 

The strategy used in the Euclidean setting relies on two key results: the first is proving the H\"{o}lder-continuity of the distribution of planes using an approximate tangent paraboloid condition (see \cite[Lemma 3.5]{ninidu1}); the second is a geometric Whitney extension lemma which exploits the H\"{o}lder-regularity result to cover pieces of the set in paraboloids with a precise number of $C^{1,\alpha}$ graphs (see \cite[Lemma 3.3]{ninidu1}).

The inherent hitches in the Heisenberg framework force the satisfactory scheme above to collapse. However, it turns out the collapse is not total: thanks to a similar inclusion result for the intersection of cylinders, we recover the H\"{o}lder-continuity result. The geometric Whitney-type argument fails in this setting mainly due to the lack of a Whitney extension theorem for maps with $\bH^n$ as the target space. At this point, we take a detour. Perhaps the miracle comes from the idea which is based on the crucial observation that in low codimensions, there is a choice of horizontal vector fields corresponding to the horizontal complements of the distribution of vertical subgroups. Moreover, we show that under our working assumptions, it is possible to make this choice in a H\"{o}lder-continuous way. This insight opens a new pathway to pull certain standard arguments and to establish our results.

\subsection{Structure of the paper}

In \Cref{sec:preliminaries}, we briefly recall the main properties of the Heisenberg group, the definition of the Hausdorff measure with some useful density results, and Pansu's differentiability for general Carnot groups. 

In \Cref{sec:rectif}, we introduce our notion of $C^{1,\alpha}$-rectifiability via regular surfaces in the sense of \Cref{Def:RegSurf}. Next, in \Cref{sec:intgras}, following \cite{MSSC10}, we define the intrinsic Grassmannian and give a characterization theorem (see \Cref{prop.str}). 

In \Cref{sec:intlemma} we prove several technical results concerning $\alpha$-paraboloids and cylinders that will be used in the proof of our main result.

Finally, in \Cref{sec:proof} we give a proof to \Cref{prop.1.2} and, after additional technical results, also \Cref{thm.1.1}.

\section{Preliminaries}\label{sec:preliminaries}

We shall restrict to the essential notions in our framework. For more details on Carnot groups and, in particular, Heisenberg groups, the reader may refer to \cites{BLU2007stratified,CDPT07}. \\

The Heisenberg group $\bH^n$ is the Carnot group whose Lie algebra $\ch^n$ admits a step two stratification
\[
\ch^{n}=\ch_{1} \oplus \ch_{2},
\]
where $\ch_{1}=\operatorname{span}\left\{X_{1}, \ldots, X_{n}, Y_{1}, \ldots, Y_{n}\right\}$ and $\ch_{2}=\operatorname{span}\{T\}$ with commutators
\[
\left[X_{i}, Y_{j}\right]= \delta_{ij} T \qquad \text{and} \qquad \left[X_{i}, X_{j}\right]=\left[Y_{i}, Y_{j}\right]=0.
\]
The vector fields $X_{1}, \ldots, X_{n}, Y_{1}, \ldots, Y_{n}$ span a vector subbundle of the tangent vector bundle $\mathrm{T}\bH^{n}$, the so-called horizontal vector bundle $\mathrm{H}\bH^{n}$. Moreover, we can identify $\bH^n$ with $\R^{2n+1}$ via exponential coordinates and we can express the group law using the {\itshape Baker-Campbell-Hausdorff formula} as follows:
\[
p \cdot q:=\left(p^{\prime}+q^{\prime}, p_{2 n+1}+q_{2 n+1}-2 \sum_{i=1}^{n}\left(p_{i} q_{i+n}-p_{i+n} q_{i}\right)\right),
\]
where $p':=\left(p_{1}, \cdots, p_{2 n}\right)$. Therefore, the inverse of $p$ is given by 
\[
p^{-1}=\left(-p^{\prime},-p_{2 n+1}\right),
\]
and $e=\mathbf{0}$ is the identity of $\bH^{n}$. The center of $\bH^{n}$ is the subgroup 
\[
\mathbb{T}:=\left\{ p \in \bH^n \: :\: p_i = 0 \text{ for all } i = 1,\ldots,2n\right\}.
\]
For any $q \in \bH^n$ and $r>0$, we denote by $\tau_{q}: \bH^n \to \bH^n$ the corresponding left-translation $p \mapsto q \cdot p=:\tau_{q}(p)$ and by $\delta_{r}: \bH^n \to \bH^n$ the dilation
\[
\delta_r(p) := (rp', r^2 p_{2n+1}).
\]
We denote by $\| \cdot \|$ the homogeneous (with respect to dilations) norm 
\[
\|p\|:= \max \left\{\left\|p'\right\|_{\R^{2n}},\left|p_{2 n+1}\right|^{1/2}\right\},
\]
where $\|\cdot\|_{\R^{2n}}$ denotes the standard euclidean norm on $\R^{2n}$, and by $d$ the corresponding metric, namely
\[
d(p, q)=d\left(q^{-1} p, e\right)=\left\|q^{-1} p\right\| \qquad \text{for all } p,q \in \bH^n.
\]
We conclude this preliminary section by recalling the definition of the Hausdorff measure in metric spaces and some density results.

\begin{definition}
Let $E \subset \bH^n$ and $k\in (0,\infty)$. The {\bfseries $k$-dimensional Hausdorff measure} of $E$, denoted by $\cH^k(E)$, is defined by setting
\[
\cH^k(E):=\sup_{\delta>0}\cH_{\delta}^k(E),
\]
where $\cH_{\delta}^k(E)=\inf \left\{ \sum_i 2^{-k}\mathrm{diam}(E_i)^k: E\subset \bigcup_i E_i, \mathrm{diam}(E_i)\le \delta\right\}$.
\end{definition}

\begin{definition}
Let $E \subset \bH^n$ be a $\cH^k$-measurable set. The upper and lower $k$-density of $E$ at any point $p\in \bH^n$ are defined, respectively, as follows:
\[ \begin{aligned}
& \Theta^{\ast k}(E, p)= \limsup_{r \to 0} \frac{\cH^k(E\cap B(p, r))}{r^k},
\\ & \Theta_\ast^k(E, p)= \liminf_{r \to 0} \frac{\cH^k(E\cap B(p, r))}{r^k}.
\end{aligned}\]
\end{definition}

We now recall some standard density estimates for Hausdorff measures which can be found, for example, in \cite[2.10.19]{Federer69}. 

\begin{lemma} \label{lemma:densities}
Let $E \subset \bH^n$ be $\cH^k$-measurable with $\cH^k(E)<+\infty$. Then
\begin{enumerate}[itemsep =0.4em, label=(\roman*)]
\item for $\cH^k$-a.e. $p \in E$, we have $2^{-k} \le \Theta^{\ast k}(E, p) \le 5^k$;
\item for $\cH^k$-a.e. $p \in\bH^n\sm E$, we have $\Theta^{\ast k}(E, p) = 0$.
\end{enumerate}
\end{lemma}

\vspace{.2cm}

\subsection*{Differentiability}
Let $\Omega$ be an open subset of $\bH^{n}$, $0< \alpha\le 1$ and $m \geq 0$ a non-negative integer. Following the notation of \cite[Section 2.2]{MSSC10}, we denote by $\mathbf{C}^{m,\alpha}(\Omega)$ the space of $m$ times continuously differentiable real-valued functions with $\alpha$-H\"{o}lder continuous $m$th-derivative in the Euclidean sense via the identification $\bH^n\equiv \mathbb{R}^{2n+1}$.

\begin{definition}
Let $f \in \mathbf{C}^{1}(\Omega)$. We define the horizontal gradient of $f$ as
\[
\nabla_{H} f:=\left(X_{1} f, \ldots, X_{n} f, Y_{1} f, \ldots, Y_{n} f\right)
\]
or, equivalently, as the section of the horizontal bundle $\mathrm{H}\bH^{n}$
\[
\nabla_{H} f:=\sum_{j=1}^{n}\left(X_{j} f\right) X_{j}+\left(Y_{j} f\right) Y_{j},
\]
with canonical coordinates $\left(X_{1} f, \ldots, X_{n} f, Y_{1} f, \ldots, Y_{n} f\right)$.
\end{definition}

\begin{definition}\label{Def:CH-regular}
A continuous function $f: \Omega \to \R$ belongs to $\mathbf{C}_{H}^{1}(\Omega)$ if the horizontal gradient $\nabla_{H} f$ exists and is continuous in $\Omega$. Furthermore, for any $0< \alpha\le 1$, if $\nabla_H f$ is $\alpha$-H\"{o}lder continuous (with respect to the homogeneous norm) i.e. 
\[
{\sup_{x,y\in \Omega, x\ne y}}\frac{\|\nabla_H f(x)-\nabla_H f(y)\|_{\R^{2n}}}{\|x^{-1}y\|} < \infty
\]
then we say that $f\in \mathbf{C}_{H}^{1,\alpha}(\Omega)$.  \\

\noindent We denote by $\left[\mathbf{C}_{H}^{1}(\Omega)\right]^{l}$ the set of $l$-tuples $f=\left(f_{1}, \ldots, f_{l}\right)$ such that $f_{i} \in \mathbf{C}_{H}^{1}(\Omega)$ for each $1 \leq i \leq l$; similarly, we also introduce the notation $\left[\mathbf{C}_{H}^{1,\alpha}(\Omega)\right]^{l}$.
\end{definition}

\begin{remark}
	The inclusion $\mathbf{C}^{1,\alpha}(\mathcal{U})\subset\mathbf{C}_{H}^{1,\alpha}(\mathcal{U})$ is strict; see, e.g., \cite[Remark 5.9]{FSSC01}.
\end{remark}

Finally, we give an intrinsic notion of differentiability in Carnot groups, which was first introduced by Pansu in \cite{Pan89}.

\begin{definition}
Let $\mathbb{G}_{1}, \mathbb{G}_{2}$ be Carnot groups and denote by $\delta_{\lambda}^{1}, \delta_{\lambda}^{2}$ the respective dilations. We say that a mapping
\[
L: \mathbb{G}_{1} \to \mathbb{G}_{2}
\]
is {\bfseries $H$-linear} if $L$ is a group homomorphism that satisfies
\[
L\left(\delta_{\lambda}^{1} (g)\right)=\delta_{\lambda}^{2} (Lg), \qquad \text{ for all }g \in \mathbb{G}_{1} \text{ and }\lambda>0.
\]
\end{definition}

\begin{definition}[Pansu differentiability]
Let $\left(\mathbb{G}_{1}, d_{1}\right)$ and $\left(\mathbb{G}_{2}, d_{2}\right)$ be Carnot groups and $\Omega \subset \mathbb{G}_{1}$. We say that a function $f: \Omega \to \mathbb{G}_{2}$ is {\bfseries Pansu differentiable} in $g \in \Omega$ if there is a $H$-linear map $L_{g}: \mathbb{G}_{1} \to \mathbb{G}_{2}$ such that
\[
\frac{d_{2}\left(f(g)^{-1} \cdot f\left(h \right), L_{g}\left(g^{-1} \cdot h \right)\right)}{d_{1}\left(g, h\right)} \to 0 \qquad \text { as } d_{1}\left(g, h \right) \to 0, \ h \in \Omega.
\]
The $H$-linear map $L_{g}$ is usually denoted by $d_{H} f_{g}$ and is referred to as the Pansu differential of $f$ at the point $g \in \mathbb G_1$.
\end{definition}

\subsection{$C^{1,\alpha}$-rectifiability in low codimension} \label{sec:rectif}

In \cite[Proposition 2.20]{MSSC10}, it was proved that the metric dimension in $\bH^n$ is given by
\[
k_m = k + 1 \qquad \text{if } n +1 \le k \le 2n.
\]
Thus, the notion of {\itshape rectifiability via Lipschitz maps} is only interesting in low dimension, i.e. $k \le n$, since, given a Lipschitz function $f : A \subset \R^k \to \bH^n$, we have
\[
\H^{k_m} \left(f(A)\right) = 0
\]
whenever dimension $k$ and metric dimension $k_m$ are not equal. Therefore, we need to find a more suitable notion of rectifiability.  The idea, looking at the Euclidean setting, is to introduce a notion of {\itshape regular surfaces} as follows:

\begin{definition}\label{Def:RegSurf}
Let $n < k \le 2n$. A set $S \subset \bH^n$ is a $k$-dimensional $(\mathbf{C}_H^{1,\alpha},\bH)$-regular surface, if for any $p \in S$ there are $\Omega \subset \bH^n$ open and $f \in [\mathbf{C}_H^{1,\alpha}(\Omega)]^{2n+1-k}$ such that
\begin{enumerate}[label=(\alph*), itemsep=1ex]
\item $d_H f_q$ is surjective for every $q \in \Omega$;
\item $S \cap \Omega = \{ q \in \Omega \: : \: f(q) = 0 \}$.
\end{enumerate}
\end{definition}

The operator $d_H f$ is the {\itshape Pansu differential} and it is represented by the horizontal gradient $\nabla_H f$ introduced above. This definition (for $\mathbf{C}_H^1$ functions) was already given in \cite{MSSC10}, so we refer the reader to that paper for more details.

\begin{definition}\label{Def:TanSub}
Let $n < k\le 2n$ and let $S$ be a $k$-dimensional $(\mathbf{C}_H^{1,\alpha},\bH)$-regular surface. The {\bfseries tangent group} to $S$ at $p \in S$, denoted as $T_{\bH} S(p)$, is given by
\[
T_{\bH} S(p) := \left\{q \in \bH^n \ : \ d_H f_{p}(q) = 0 \right\}.
\]
\end{definition}

The following characterization of $(\mathbf{C}_H^{1,\alpha},\bH)$-regular surfaces is an immediate consequence of the definition above:

\begin{proposition} \label{prpps23}
A set $S$ is a $k$-dimensional $(\mathbf{C}_H^{1,\alpha},\bH)$-regular surface if and only if $S$ is locally the intersection of $(2n+1-k)$ $1$-codimensional $(\mathbf{C}_H^{1,\alpha},\bH)$-regular surfaces with linearly independent normal vectors.
\end{proposition}

\begin{remark}
Let $\Omega\subset\bH^n$ be an open set. The Taylor's expansion of a function $f\in \mathbf{C}_{H}^{1,\alpha}(\Omega)$ based at the point $x_0\in\Omega$ is given by (see \cite[Theorem~1.42]{FS82})
\begin{equation}\label{TaylorExp}
   f(x)=f(x_0)+d_{H}f_{x_0}(x_0^{-1}x) + \mathcal{O}\left(d(x_0,x)^{1+\alpha}\right).
\end{equation}
\end{remark}

To conclude this preliminary section, we now give the formal definition of $C^{1,\alpha}$-rectifiability for a subset of the Heisenberg group $\bH^n$.

\begin{definition} \label{def:rec}
A measurable set $E \subset \bH^n$ is {} $C^{1,\alpha}$-rectifiable if there are $k$-dimensional $(\mathbf{C}_H^{1,\alpha},\bH)$-regular surfaces $S_i$, with $i \in \N$, such that
\[
\cH^{k_m} \left(E\sm\bigcup_{i\in\N}S_i\right)=0,
\]
where $k_m=k$ if {} $1\le k \le n$ and $k_m=k+1$ if {} $n+1\le k \le 2n$.
\end{definition}

\subsection{Whitney's extension theorem} The following Whitney-type extension theorem was proved in \cite[Theorem 4]{vodpup} for general Carnot groups, but for simplicity here we only state it for the Heisenberg group $\bH^n$.

\begin{theorem}[$C^{1,\alpha}$-extension]\label{Lemma:Whitney}
Let $F$ be a closed subset of $\bH^n$, $\alpha\in(0,1]$ and $f:F \to \R$, $g:F \to \mathrm{H}\bH^n$ functions satisfying the following property: there exists a positive constant $M$ such that
\begin{enumerate}[label=(\roman*),itemsep=.2em]
    \item $|f(x)| \le M$ and $|g(x)| \le M$ on every compact subset of $F$;
    \item $|f(x)-f(y)-\langle g(x),\pi(y^{-1}x)\rangle|\le Md(x,y)^{1+\alpha}$ for every $x,y\in F$;
    \item $|g(x)-g(y)|\le M d(x,y)^\alpha$ for every $x,y\in F$;
\end{enumerate}
where $\langle\cdot,\cdot\rangle$ denotes the inner product in $\mathrm{H}\bH^n$. Then there exists $\tilde{f}:\bH^n\to \R$, extending $f$ to $\bH^n$, such that $\tilde{f}\in \mathbf{C}^{1,\alpha}_H(\bH^n)$ and
\[
g(x)=\nabla_H\tilde{f}(x) \qquad \text{for every } x \in F.
\]
\end{theorem}

\subsection{The intrinsic Grassmannian}\label{sec:intgras}

A subgroup $S \subset \bH^n$ is a {\itshape homogeneous subgroup} if, for every $r > 0$, we have
\[
\delta_r(S) \subseteq S.
\]

\begin{definition}\label{def:horizontal}
Let $S \subset \bH^n$ be a homogeneous subgroup. We say that \mbox{}
\begin{itemize}
	\item $S$ is {\itshape horizontal} if it is contained in $\exp(\ch_1)$; \smallskip
	\item $S$ is {\itshape vertical} if it contains the center $\mathbb T$ of $\bH^n$.
\end{itemize}
\end{definition}

\noindent We also notice that horizontal subgroups are commutative while vertical subgroups are non-commutative and {\em normal} in $\bH^n$.

\begin{definition}
We say that two homogeneous subgroups $S$ and $T$ of $\bH^n$ are {\itshape complementary subgroups} in $\bH^n$ if the following holds: 
\[
S \cap T = \{0\} \qquad \text{and} \qquad \bH^n = T \cdot S.
\]
In addition, if $T$ is normal, we say that $\bH^n$ is the {\itshape semidirect} product of $S$ and $T$ and we denote it by $\bH^n = T \rtimes S$.
\end{definition}

If $\bH^n$ is the semidirect product of homogeneous subgroups $S$ and $T$, then we can define unique projections $\pi_S : \bH^n \to S$ and $\pi_T : \bH^n \to T$ such that
\[
\mathrm{id}_{\bH^n} = \pi_T \cdot \pi_S.
\]
Moreover, if $T$ is normal in $\bH^n$, we have the following identities: \mbox{}
\begin{itemize}[label=\tiny$\bullet$ ]
	\item $\pi_S(p^{-1}) = \pi_S^{-1}(p)$ and $\pi_T(p^{-1}) = \pi_S^{-1}(p) \cdot \pi_T^{-1}(p) \cdot \pi_S(p)$; \smallskip
	\item $\pi_S(\delta_r p) = \delta_r \pi_S(p)$ and $\pi_T(\delta_r p) = \delta_r \pi_T(p)$; \smallskip
	\item $\pi_S(p \cdot q) = \pi_S(p) \cdot \pi_S(q)$ and $\pi_T(p \cdot q)=\pi_T(p) \cdot \pi_S(p)\cdot  \pi_T(q) \cdot \pi_S^{-1}(p)$.
\end{itemize}

\begin{proposition} \label{prop:idsl}
If $\bH^n = T \rtimes S$, then the projections $\pi_S$ and $\pi_T$ are continuous, $\pi_S$ is $H$-linear, and there is a constant $c(S,T) := c > 0$ such that
\[
\begin{aligned}
& c \|\pi_S(p)\| \leq d(p, T) \leq \|\pi_S(p)\|,
\\ & c \|\pi_S^{-1}(p) \cdot \pi_{T}(p) \cdot \pi_S(p)\| \leq d(p, S) \leq \|\pi_S^{-1}(p) \cdot \pi_{T}(p) \cdot \pi_S(p)\|, \end{aligned}
\]
for all $p \in \bH^n$, where we use the notation
\[
d(p,S) := \inf_{s \in S} d(p,s) = \inf_{s \in S} \| p^{-1}s\|.
\] 
\end{proposition}

This result was proved in \cite{MSSC10} for the Heisenberg group and generalized in \cite{idumama1} to all homogeneous groups.

\begin{remark}
In \cite{idumama1}, with Magnani, we also showed that the constant $c$ does not depend on $S$ and $T$: it suffices to consider , for $1 \le k \le n$, a $k$-homogeneous subgroup $S$ and write 
\[
\bH^n = S^\perp \rtimes S,
\]
where $S^\perp$ is the vertical subgroup defined as follows: If $S = \langle f_1,\ldots,f_k\rangle$, we take
\[
S^\perp := \langle f_1,\ldots,f_k\rangle^{\perp_{H^1}} \oplus \langle e_{2n+1} \rangle,
\]
where $\perp_{H^1}$ denotes the orthogonal to the horizontal layer of $\bH^n$ with respect to the fixed scalar product. In this case, we denote by $c_\bH$ the universal constant.
\end{remark}

We are now ready to provide the notion of {\em intrinsic Grassmannian}, which was first introduced in \cite{MSSC10} for the Heisenberg groups.

\begin{definition}\label{def.grass}
Let $k \in \{0,\ldots,2n+1\}$. A $k$-homogeneous subgroup $S$ belongs to the $k$-Grassmannian, denoted by $\cG(\bH^n,k)$, if there is a $(2n+1-k)$-subgroup $T$ such that $\bH^n = T \cdot S$. Moreover, the corresponding union
\[
\cG(\bH^n) := \bigcup_{k = 0}^{2n+1} \cG(\bH^n,k)
\]
is often referred in the literature as the {\em intrinsic Grassmannian} of $\bH^n$.
\end{definition}

\begin{proposition}\label{prop.str}
The trivial subgroups $\{e\}$ and $\bH^n$ are the unique elements of $\cG(\bH^n,0)$ and $\cG(\bH^n,2n+1)$ respectively. Moreover, the following holds: \mbox{}
\begin{enumerate}[label=(\roman*)]
\item for $1 \le k \le n$, the $k$-Grassmannian $\cG(\bH^n,k)$ coincides with the set of all horizontal $k$-homogeneous subgroups; \smallskip
\item for $n < k \le 2n$, the $k$-Grassmannian $\cG(\bH^n,k)$ coincides with the set of all vertical $k$-homogeneous subgroups.
\end{enumerate}
Consequently, any vertical subgroup $T$ with linear dimension $1\le k \le n$ is not an element of the intrinsic Grassmannian.
\end{proposition}

This result was proved in \cite[Proposition 2.17]{MSSC10}. Notice that, if $\G$ is a homogeneous group, this characterization is no longer true since the identity
\[
\cG(\G,k) = \left\{ \text{vertical $k$-homogeneous subgroups} \right\}
\]
does not hold for every $k$, unless we put additional assumptions on $\G$ or limit the possible values that $k$ can take (see \Cref{sec:carnot}).

\begin{remark}
If $S$ is a $(\mathbf{C}_H^{1,\alpha},\bH)$-regular surface, then $T_\bH S(p) \in \cG(\bH^n)$.
\end{remark}

\begin{remark} 
The intrinsic Grassmannian $\cG(\bH^n)$ is a subset of the Euclidean one, so it can be endowed with the subspace topology. Moreover, it is easy to verify that $\cG(\bH^n,k)$ is a compact metric space with respect to the distance
\[
\rho(S_1, S_2):= \max_{\|x \|=1} d\left( \pi_{S_1}(x),\pi_{S_2}(x)\right).
\]
\end{remark}

\section{Elementary  Geometric lemmas} \label{sec:intlemma}

The main objects used in this paper are $\alpha$-paraboloids and cylinders, so we now recall the correct definitions in our framework.

\begin{definition}\label{def.objects}
Fix $\alpha \in (0,1]$, $\lambda,\eta >0$ and $r>0$. The {\itshape $\alpha$-paraboloid} centered at $x\in \bH^n$ with base $S \in \cG(\bH^n)$ and parameter $\lambda$ is defined as
\[
Q_\alpha(x,S,\lambda) := \left\{ y \in \bH^n \ : \ d(x^{-1}y,S) \le \lambda d(x,y)^{1+\alpha} \right\}.
\]
On the other hand, the {\itshape cylinder} with axis $S$ and parameter $\eta$ is given by
\[
\cC(S, \eta):=\left\{y\in \bH^n \ : \ d(y,S)<\eta \right\}.
\]
\end{definition}

We are now ready to introduce approximate tangent paraboloids, which essentially are the objects that characterize the $C^{1,\alpha}$-rectifiability in \Cref{thm.1.1}.

\begin{definition}\label{Def:AppTanPar}
Let $E \subset \bH^n$ be a $\H^{k_m}$-measurable set and $\alpha \in (0,1]$. A homogeneous subgroup $V_p$, of dimension $k$ and metric dimension $k_m$, is an {\itshape approximate tangent paraboloid} to $E$ at a point $p$ if \mbox{} 
\begin{enumerate}[label=(\roman*)]
	\item $\Theta^{\ast k_m}(E,p)>0$; \smallskip
	\item $\lim_{r \to 0} r^{-k_m} \H^{k_m} \left(E \cap B(p,r) \sm Q_\alpha(p,V_p,\lambda)\right) = 0$ for every $\lambda > 0$.
\end{enumerate}
We write $\mathrm{apPar}_{\bH}^{k_m}(E,p)$ for the set of all approximate tangent paraboloids to $E$ at $p$ and, when it is unique, we simply denote it by $V_p$.
\end{definition}

The following result, which is a consequence of the definitions, gives the analogous relationship between estimates on cylinders and paraboloids as in \cite[Lemma 2.3]{ninidu1} in the Euclidean setting. The proof in this setting is essentially the same thus we omit it.

\begin{lemma}\label{Lem:EquivConePar}
Let $V \in \cG(\bH^n,k)$ and fix $r_0 > 0$. Suppose that
\[
\H^{k_m} \left(E \cap B(x,r) \setminus \cC(V,\lambda r^{1+\alpha})\right) \le \epsilon r^{k_m}
\]
for every $r < r_0$. Then we have
\[
\H^{k_m} \left(E \cap B(x,r) \setminus Q_\alpha(x,V,\lambda')\right) \le \frac{\epsilon}{1-2^{-{k_m}}} r^{k_m}
\]
for every $r<r_0$, where $\lambda' := 4^{1+\alpha} \lambda$.
\end{lemma}


Using the Taylor's expansion \eqref{TaylorExp} and following closely the strategy proposed in \cite[Lemma 2.28]{MSSC10}, we obtain an useful inclusion result:

\begin{lemma}\label{ConvGeomLem}
Fix $n<k\le 2n$. Let $S\subset \bH^n$ be a $k$-dimensional $(\mathbf{C}_H^{1,\alpha},\bH)$-regular surface and $p\in S$. Then there exist $\lambda>0$ and $r_0:=r_0(S,p)>0$ such that
\begin{equation}\label{Eq:ConvGeomLem}
S\cap B(p,r_0)\subset Q_\alpha(p,T_{\bH}S(p),\lambda).
\end{equation}
\end{lemma}

\begin{proof}
By \Cref{Def:RegSurf} and \Cref{Def:TanSub}, there are $r_0>0$ and $f\in\left[\mathbf{C}_{H}^{1,\alpha}(\Omega)\right]^{2n+1-k}$ such that $d_Hf_p:\bH^n\rightarrow \R^{2n+1-k}$ is surjective, 
\[
S\cap B(p,r)=\left\{ q \in \Omega \ : \ f(q)=0 \right\} \qquad \text{and} \qquad T_{\bH}S(p)=\ker(d_Hf_p).
\]
For any $q\in S\cap B(p,r_0)$, using \eqref{TaylorExp} we have that
\begin{equation}\label{H1}
\big\|d_Hf_p(p^{-1}q)\big\|_{\R^{2n+1-k}}=\mathcal{O}\left(d(p,q)^{1+\alpha}\right),
\end{equation}
while by $H$-linearity of $d_Hf_p$ there is $c:=c(x,f)>0$ such that
\begin{equation}\label{H2}
\big\|d_Hf_p(p^{-1}q)\big\|_{\R^{2n+1-k}}\ge c\,d\left(p^{-1}q,T_{\bH}S(p)\right).
\end{equation}
Indeed, if $L:\bH^n\to \R^{2n+1-k}$ is $H$-linear, then $\ker(L)$ is a vertical subgroup and, by the intrinsic decomposition, there exists $V$ horizontal such that
\[
\ker(L)\cdot V=\bH^n.
\] 
Then $L:V\to\R^{2n+1-k}$ is injective and thus there exists $C>0$ such that
\[
\|L(v)\|_{\R^{2n+1-k}}\ge C\|v\| \qquad \text{for all } v \in V.
\]
Finally, by \eqref{H1} and \eqref{H2} we can find $\lambda>0$ such that \eqref{Eq:ConvGeomLem} holds.
\end{proof}

We now prove that vertical subgroups in the Grassmannian have horizontal complements that can be chosen in a Lipschitz-continuous way.

\begin{definition}
	Let $\nu \in \ch_1$. We denote by $\N(\nu)$ the $1$-codimensional normal subgroup orthogonal to $\nu$, namely
	\[
	\N(\nu) := \left\{ p \in \bH^n \: : \: \langle \nu, \pi(p) \rangle = 0 \right\},
	\]
	where $\pi$ is the projection of $\bH^n$ onto the first layer $\ch_1$, defined as follows:
	\[
	\pi(p) := \sum_{i = 1}^n (p_i X_i + p_{n+i}Y_i).
	\]
\end{definition}

\begin{lemma} \label{lemma.2.32}
Let $T \in \cG(\bH^n,k)$ for $n < k \le 2n$. Then we can find unit vectors $\nu_1,\ldots, \nu_{2n+1-k} \in \ch_1$ such that
\[
S := \exp \left( \mathrm{span}\{\nu_1,\ldots,\nu_{2n+1-k}\} \right)
\]
is a horizontal complement of $T$ and the choice of the unit vectors is Lipschitz-continuous with respect to $T$.

Furthermore $T = \cap_{i} \N(\nu_i)$ and for all $p \in \bH^n$ and all $\alpha \in (0,1]$, the following inclusion holds:
\begin{equation}\label{eq:parabinclusion}
Q_\alpha(p,T,\lambda) \subseteq \bigcap_{i= 1}^{2n+1-k} Q_\alpha \left(p,\N(\nu_i),\lambda \right).
\end{equation}
\end{lemma}

\begin{proof}
By \cite[Lemma 3.26]{FSSC6}, we can always find unit vectors $\nu_1,\ldots,\nu_{2n+1-k} \in \ch_1$ such that the subgroup
\[
S := \exp \left( \mathrm{span}\{\nu_1,\ldots,\nu_{2n+1-k}\} \right)
\]
is a horizontal complement of $T$ and, as mentioned in \cite[Lemma 2.32]{MSSC10}, the choice of the unit vectors is continuous with respect to $T$. This immediately implies Lipschitz-continuity by the linearity of the argument used to prove \cite[Lemma 3.26]{FSSC6}. 

Now denote by $\mathfrak{t}\subset\ch$ the Lie algebra of $T$ and, for every $i$, consider the Lie algebra obtained by removing $\nu_i$, namely
\[
\ch^i:=\mathrm{Span}\left\{\mathfrak{t},\nu_1,\ldots,\nu_{i-1},\nu_{i+1},\ldots,\nu_{2n+1-k}\right\}.
\]
It follows that $\N(\nu_i)=\exp(\mathfrak{h}^i)\in \cG(\bH^n,2n)$ and $T=\cap_i\N(\nu_i)$. Finally, since 
\[
d(p,\N(\nu_i))\le d(p,T),
\] 
we have that $Q_\alpha(p,T,\lambda)\subset Q_\alpha(p,\N(\nu_i),\lambda)$ for all $i=1,\ldots,2n+1-k$, from which the inclusion \eqref{eq:parabinclusion} easily follows.
\end{proof}

The following result tells us that the intersection of two cylinders can be included in a cylinder with axis of {\itshape strictly lower} dimension. The proof is similar to the Euclidean framework, which can be found in \cite[Lemma 2.1]{ninidu1}.

\begin{lemma}\label{lemma:tube}
Let $S,T \in \cG(\bH^n,k)$ with $n < k \le 2n$ and set $\vartheta:=\rho(S,T)$. Then there are $Z \in \cG(\bH^n,k-1)$ and $\ell > 0$ such that for every $\eta>0$ we have
\[
\cC\left(S, \eta\right)\cap \cC\left(T, \eta\right)\subseteq \cC\left(Z, \frac{3n\eta}{\ell \vartheta}\right).
\]
\end{lemma}

\begin{proof} 
First, we claim that there is $e\in T^\perp$ with $\|e\|=1$ such that $\|\pi_{S^\perp} (e)\|= \ell \vartheta$. Indeed, the triangular inequality gives
\[
\vartheta \ge \rho(S^\perp,T^\perp) - C_T - C_S \ge \sup_{ \substack{t \in T^\perp \\[.2em] \|t\|=1}} \|\pi_{S^\perp}(t)\| - C_T - C_S,
\]
where the (positive and finite) constants $C_T$ and $C_S$ are defined as
\[
C_T = \max_{\|p\|=1} d\left(\pi_T(p),\pi_{T^\perp}(p) \right) \qquad \text{and} \qquad C_S = \max_{\|p\|=1} d \left(\pi_S(p),\pi_{S^\perp}(p)\right).
\]
It follows that
\[
\vartheta + C_T + C_S \ge \sup_{ \substack{t \in T^\perp \\[.2em] \|t\|=1}} \|\pi_{S^\perp}(t)\| > 0,
\]
which means that there must be constants $L_1,L_2 > 0$ such that
\[
L_1 \vartheta \ge \sup_{ \substack{t \in T^\perp \\[.2em] \|t\|=1}} \|\pi_{S^\perp}(t)\| \ge L_2 \vartheta.
\]
By compactness, we can find $e \in T^\perp$ and $\ell \in [L_2,L_1]$ such that
\[
\|\pi_{S^\perp} (e)\| = \ell \vartheta,
\]
and this proves our claim. Now consider an orthonormal basis $e_{k+1},\ldots,e_{2n+1}$ of $S^\perp$ and define the $(k-1)$-dimensional subspace
\[
Z:=\mathrm{span}\{e,e_{k+1},\ldots,e_{2n+1}\}^\perp.
\]
For $p\in \cC(S, \eta)\cap \cC(T,\eta)$, by \Cref{prop:idsl} we have
\[
d(p,S) < \eta \implies \| \pi_{S^\perp}(p) \| < \eta \qquad \text{and} \qquad d(p,T) < \eta \implies \| \pi_{T^\perp}(p) \| < \eta,
\]
which means (taking into account that the $e_i$'s give a basis of $S^\perp$) that
\[
\|p^{-1}\cdot e_i\|\le \eta \quad \text{for every }i=k+1,\ldots, 2n+1 \qquad \text{and} \qquad \|p^{-1}\cdot e\|\le \eta.
\]
Now set $e':=\tfrac{\pi_{S} (e)}{\|\pi_{S} (e)\|}$ and consider the resulting orthonormal basis $\{e',e_{k+1},\ldots, e_{2n+1}\}$ of $Z^\perp$. Then, for $p \in \cC(S, \eta)\cap \cC(T, \eta)$, the triangular inequality gives
\[ \begin{aligned}
\|p^{-1}\cdot e'\|& =\frac{1}{\|\pi_{S} (e)\|} \|p^{-1}\cdot \pi_{S}(e)\|
\\ & =\frac{1}{\ell \vartheta} \left\|p^{-1}\cdot \left(e-\sum_{i=k+1}^{2n+1} (e^{-1}\cdot e_i)e_i\right)\right\|
\\ & \leq \frac{1}{\ell \vartheta}(2n+1-k+1)\eta.
\end{aligned} \]
As a consequence, we have
\[ \begin{aligned}
\|\pi_{Z^\perp}(p) \| & \le \|p^{-1}\cdot e'\|+\sum_{i=k+1}^{2n+1} \|p^{-1}\cdot e_i\|
\\ & \le  \frac{1}{\ell \vartheta}(2n+1-k+1)\eta +\sum_{i=k+1}^{2n+1} \|p^{-1}\cdot e_i\|
\\ & \le \frac{1}{\ell \vartheta} (2(2n+1)-2k+1)\eta\leq \frac{3n}{\ell \vartheta}\eta,
\end{aligned} \]
and this concludes the proof.
\end{proof}

\section{Proof of the main results} \label{sec:proof}

The goal of this section is to prove \Cref{thm.1.1} and \Cref{prop.1.2}. Indeed, the latter follows by adapting the argument in \cite[Section 3]{ninidu1} to Heisenberg groups:


\begin{proof}[Proof of \Cref{prop.1.2}]
Let $E \subset \bH^n$ be $C^{1,\alpha}$-rectifiable and let $\{ \Gamma_i \}_{i \in \N}$ be the family of $(\mathbf{C}_H^{1,\alpha},\bH)$-regular surfaces such that
\[
\H^k\left(E \sm \bigcup_{i\in \mathbb{N}}\Gamma_i\right)=0.
\]
Then $E$ is $\cH^k$-rectifiable; so, by \cite[Theorem 3.15]{MSSC10}, for $\cH^k$-a.e. $p \in E$ there exists an approximate tangent subgroup $T_p \in \cG(\bH^n,k)$ and $\Theta_\ast^k(E,p)>0$.

For every $i\in \N$, denote by $E_i$ the set $E\cap \Gamma_i$. By \Cref{lemma:densities} for $\H^k$-a.e. $p\in E_i$, we have that the density satisfies
\begin{equation}\label{eq:E_i}
    \Theta^k \left(E\setminus E_i,p\right)=0. 
\end{equation}
Moreover, by \Cref{ConvGeomLem} there is $\lambda>0$ such that
\[
E_i\cap B(p,r)\sm Q_\alpha(p,T_p,\lambda)=\varnothing
\] 
for every $p\in E_i$ and $r>0$ small enough. Finally, the conclusion follows by the latter and \eqref{eq:E_i}.
\end{proof}

\subsection{Technical preliminaries}

To prove \Cref{thm.1.1}, we first need to obtain some technical lemmas. Recall that we are assuming, for $\cH^{k_m}$-a.e. $p \in E$, the existence of $V_p \in \cG(\bH^n,k)$ and $\lambda > 0$ such that
\[
\lim_{r \to 0^+} r^{-k_m} \cH^{k_m} \left(E \cap B(p,r) \sm Q_\alpha(p,V_p,\lambda)\right) = 0.
\]
Therefore, in the following we will often identify $p$ with the unique vertical subgroup $V_p$ that satisfies this property. More precisely, we use the notation
\[
C_\alpha^r(p) := \cC(V_p,\lambda r^{1+\alpha})
\]
for cylinders with axis $V_p$. Notice that the parameter $\lambda>0$ does not appear on the left-hand side because it has already been fixed above.

\begin{lemma}\label{lemma:keyHn}
Let $E\subset\bH^n$, $n < k \le 2n$, and let $M,\lambda,\delta>0$, $0<\alpha,\,r\le 1$ be fixed. Suppose that for every $z\in E$ and for every $s>0$ we have 
\begin{equation}\label{eq:upperhypHn}
\H^{k_m} \left(E\cap B(z,s)\right) \le M s^{k_m}.
\end{equation}
Let $p,q$ be any two points such that $d(p,q) \le r$ and $V_p,V_q \in \cG(\bH^n,k)$ satisfy
\begin{equation}\label{eq:lowerboundHn}
\begin{cases}
\H^{k_m} \left( E \cap B(p, r) \right) \ge \delta  r^{k_m} \\
\H^{k_m} \left( E \cap B(q, r) \right) \ge \delta  r^{k_m}
\end{cases}
\end{equation}
and
\begin{equation}\label{eq:upperboundHn}
\begin{cases}
\H^{k_m} \left( E \cap B(p, 2r)\setminus C_\alpha^{r}(p) \right) \le \eps  r^{k_m} \\
\H^{k_m} \left(E\cap B(q, 2r)\setminus C_\alpha^{r}(q)\right)\le \eps  r^{k_m}
\end{cases}
\end{equation}
where $\eps \le \tfrac{\delta}{4}$. Then there exists a positive constant $C := C(n,\delta,M,\lambda)$ such that
\[
\rho(V_p,V_q)\leq C r^\alpha.
\]
\end{lemma}

\begin{proof}
We argue by contradiction. Suppose that we have 
\[
\vartheta:=\rho(V_p,V_q)>Cr^\alpha,
\] 
with $C=C(n,\delta,M,\lambda)>0$ to be chosen later. First, we observe that by the assumptions \eqref{eq:lowerboundHn} and \eqref{eq:upperboundHn} we have the estimate
\[
\cH^k(E\cap C_\alpha^r(p)\cap B(p,r))\geq (\delta-\eps)r^k.
\] 
Furthermore, using the trivial inclusion  
\[
E\cap C_\alpha^r(p)\cap B(p,r)\sm C_\alpha^r(q)\subseteq  E\cap B(q,2r)\setminus C_\alpha^r(q),
\]
it follows that $\cH^k(E\cap C_\alpha^r(p)\cap B(p,r)\sm C_\alpha^r(q))\leq \eps r^k$, and thus
\[
(\delta-\eps)r^k\leq \cH^k\left(E\cap C_\alpha^r(p)\cap B(p,r)\right)\le \cH^k\left(E \cap C_\alpha^r(p)\cap B(p,r)\cap C_\alpha^r(q)\right) + \eps r^k.
\]
Finally, the assumption $\eps \le \delta/4$ implies the estimate
\begin{equation}\label{eq:positive-mass}
\cH^{k_m}\left(E\cap C^r_\alpha(p)\cap C^r_\alpha(q)\cap B(p,r)\right)\ge \frac{\delta}{2}r^{k_m},
\end{equation}
which, in turn, shows that $B(p,r)\cap C^r_\alpha(p)\cap C^r_\alpha(q)\ne \varnothing$.

Now, since $V_p$ and $V_q$ are vertical $k$-subgroups, they intersect; hence, by \Cref{lemma:tube} we have that there exists a vertical $(k-1)$-subgroup $Z$ such that
\[
C^r_\alpha(p)\cap C^r_\alpha(q) \subset \cC\left(Z,\frac{4n\lambda}{\ell\vartheta} r^{1+\alpha}\right),
\]
where $\ell>0$ is given by \Cref{lemma:tube}. We set
\[
\eta:=\frac{4n\lambda}{\ell \vartheta} r^{1+\alpha} \qquad \text{and} \qquad E_\eta:=E\cap B(p,r)\cap \cC(Z,\eta).
\]
Notice that we can assume $\eta \le r$, since otherwise we would have
\[
\eta > r \implies \vartheta < \frac{4n\lambda}{\ell} r^\alpha,
\] 
which is not possible since we have $\vartheta>Cr^\alpha$. We now claim that 
\begin{equation}\label{eq:upperboundClaim}
\H^{k_m}(E_\eta)<\frac{4n \lambda C_1 M}{\ell C}r^{k_m},
\end{equation}
where $C_1=C_1(n,k_m)>0$. Indeed, clearly, the set $E_\eta$ can be covered with $h$ balls of radius $\eta r$, where $h\le C_1\eta^{-k_m+1}$  and $C_1$ depends only on $n$ and $k_m$. Thus, by \eqref{eq:upperhypHn} we have, using that $0<r<1$ and $\vartheta>Cr^\alpha$, that
\[
\H^{k_m}(E_\eta) \le C_1M\eta r^{k_m}=\frac{4n \lambda C_1M r^{1+\alpha}}{\ell\theta}r^{k_m} < \frac{4n\lambda C_1M}{\ell C}r^{k_m},
\]
and this proves the claim \eqref{eq:upperboundClaim}. Finally, using the estimate above together with \eqref{eq:positive-mass}, yields the following inequality:
\[
\frac{\delta}{2}r^{k_m}\le \H^{k_m}(E_\eta)< \frac{4n \lambda C_1M}{\ell C}r^{k_m}.
\]
Choosing $C:=[4n/(\delta\ell)] \max\left\{C_1M\lambda,\lambda\right\}$ immediately yields to a contradiction, concluding the proof of the lemma.
\end{proof}

A simple application of this result allows us to slightly improve the statement of \Cref{lemma.2.32}, obtaining the local $\alpha$-H\"older continuity of the map $p\mapsto V_p$.

\begin{lemma} \label{lemma.2.32.holder}
Let $n < k \le 2n$ and fix $M,\lambda,\delta>0$, $0<\alpha,\,r\le 1$. Let $\Omega\subset\bH^n$ be an open subset with $\operatorname{diam}(\Omega)>2r$ and $E\subseteq \Omega$ as in \Cref{lemma:keyHn}. Consider
\[
\Omega \ni p \longmapsto V_p \in \cG(\bH,k),
\]
and assume that for every $p,q\in E$ with $d(p,q)<r$, the conditions \eqref{eq:lowerboundHn}-\eqref{eq:upperboundHn} hold. If we denote by
\[
W_p = \operatorname{span}\left( \exp \{\nu_1(p),\ldots,\nu_{2n+1-k}(p)\} \right)
\]
the horizontal complement of $T:=V_p$ in \Cref{lemma.2.32}, then the mappings
\[
p \longmapsto \nu_j(p) \qquad \text{for } j = 1,\ldots,2n+1-k,
\]
are $\alpha$-H\"older continuous in some $\Omega'\subset \Omega$.
\end{lemma}


\begin{proof}
By \Cref{lemma:keyHn}, for every $p,q\in E$ with $d(p,q)<r$ we have
\[
\rho(V_p,V_q) \leq C d(p,q)^\alpha
\]
for some $C>0$. On the other hand, in \Cref{lemma.2.32} we proved that
\[
V_p \longmapsto \nu_j(p) \qquad \text{for } j=1,\ldots,2n+1-k
\]
are Lipschitz-continuous maps in an open set $\tilde \Omega \subset \Omega$ with $p \in \tilde \Omega$. In other words, there exists a positive constant $c'$ such that 
\[
d(\nu_j(p),\nu_j(q)) \leq c' \rho(V_p,V_q) \qquad \text{for all } p,q \in \tilde \Omega.
\]
Putting the inequalities above together, we get
\[
d(\nu_j(p),\nu_j(q)) \le \tilde{c} d(p,q)^\alpha \qquad \text{for all } p, q \in \tilde \Omega \cap E,
\]
where $\tilde{c} := C \cdot c'$. The conclusion follows by taking $\Omega':=\tilde{\Omega}\cap E$.
\end{proof}

\begin{lemma} \label{lemma.2.29}
Let $\alpha \in (0,1]$, $\lambda>0$, $p,q \in \bH^n$ and $V_p, V_q \in \cG(\bH^n,k)$ vertical subgroups satisfying the assumptions of \Cref{lemma:keyHn}. Then there exists $\lambda' > \lambda$ such that, if we also suppose
	\[
	q \notin Q_\alpha(p,V_p,\lambda'),
	\]
	then the following inclusion holds:
	\[
	B\left( q, r/2 \right) \cap C_\alpha^{r/2} \left(q\right) \subset B(p,2r) \setminus Q_\alpha \left(p,V_p, \lambda \right),
	\]
	where $r := d(p,q)$. Moreover, we can take any $\lambda'$ satisfying the inequality
	\beq\label{lambdapvalue}
	\lambda' \ge \frac{2^\alpha cC + (1+3^{1+\alpha})\lambda}{2^{1+\alpha}c},
	\eeq
	where $c$ and $C$ are the constants given in \Cref{prop:idsl} and \Cref{lemma:keyHn}.
\end{lemma}

\begin{proof}
	Let $z \in B(q, r/2) \cap C_\alpha^{r/2}(q)$. Using the triangular inequality, we have
	\[
	\left[d(p,z)\right]^{1+\alpha} \leq \left[ d(p,q) + d(q,z) \right]^{1+\alpha} \leq \left( \frac{3}{2} \right)^{1+\alpha} r^{1+\alpha}.
	\]
	Therefore, to prove that $z$ does not belong to the paraboloid $Q_\alpha \left(p,V_p, \lambda \right)$, we only need to estimate $d(p^{-1}z,V_p)$ from below. We start by noticing that
	\beq \label{eq.k1}
	q \notin Q_\alpha(p,V_p,\lambda') \implies d(p^{-1}q,V_p) > \lambda' r^{1+\alpha},
	\eeq
	and, similarly, also that
	\beq \label{eq.k2}
	z \in C_\alpha^{r/2}(q) \implies d(q^{-1}z,V_q) \le \lambda \left( \frac{r}{2} \right)^{1+\alpha}.
	\eeq
	Since $V_p$ is a vertical homogeneous subgroups, by \Cref{prop:idsl} we have
	\[
	d(p^{-1}z,V_p) \ge c \| \pi_{V_p^\perp}(p^{-1}z)\|,
	\]
	where $c$ is a constant that only depends on the dimension of $V_p$. Now, taking into account that $V_p^\perp$ and $V_q^\perp$ are horizontal subgroups, we have
	\[ \begin{aligned}
		\| \pi_{V_p^\perp}(p^{-1}z)\| & = \left\| \pi_{V_p^\perp}(p^{-1}q) \cdot \pi_{V_q^\perp}(q^{-1}z) \cdot \pi_{V_q^\perp}^{-1}(q^{-1}z)\cdot \pi_{V_p^\perp}(q^{-1}z)\right\|
		\\ &  \ge \|  \pi_{V_p^\perp}(p^{-1}q) \| - \|  \pi_{V_q^\perp}(q^{-1}z) \| -  \| \pi_{V_q^\perp}^{-1}(q^{-1}z)\cdot \pi_{V_p^\perp}(q^{-1}z) \|.
	\end{aligned}\]
	The last term can easily be bounded by noticing that
	\[
	\| \pi_{V_q^\perp}^{-1}(q^{-1}z)\cdot \pi_{V_p^\perp}(q^{-1}z) \|  \le \rho(V_p,V_q) d(q,z)  \le (C r^\alpha) \frac{r}{2} = \frac{C}{2} r^{1+\alpha},
	 \]
	where $C$ is the constant given in \Cref{lemma:keyHn}. If we now apply \Cref{prop:idsl} to the first two terms, we get the following inequality:
	\[ \begin{aligned}
		d(p^{-1}z,V_p) & \ge c \| \pi_{V_p^\perp}(p^{-1}z)\|
		\\ & \ge c \|  \pi_{V_p^\perp}(p^{-1}q) \| - c \|\pi_{V_q^\perp}(q^{-1}z)\| - c\|\pi_{V_q^\perp}^{-1}(q^{-1}z)\cdot \pi_{V_p^\perp}(q^{-1}z) \|
		\\ & \ge c d(p^{-1}q,V_p) - d(q^{-1}z,V_q) - \frac{C}{2}c r^{1+\alpha},.
	\end{aligned} \]
	Therefore, using both \eqref{eq.k1} and \eqref{eq.k2}, leads to
	\[
	d(p^{-1}z,V_p)  >\left[ c \lambda'  - \lambda 2^{-1-\alpha} - c \frac{C}{2} \right] r^{1+\alpha}
	\]
	and, as a consequence, the quantity
	\[
	\left[ c \lambda'  - \lambda 2^{-1-\alpha} - c \frac{C}{2} \right] r^{1+\alpha} - \lambda \left( \frac{3}{2} \right)^{1+\alpha} r^{1+\alpha}
	\]
	is non-negative if we take $\lambda'>\lambda$ exactly as in \eqref{lambdapvalue}, concluding the proof.
\end{proof}

\subsection{Proof of Theorem \ref{thm.1.1}} Following the strategy proposed in \cite{MSSC10}, we first need the uniqueness a.e. of approximate tangent paraboloids.

\begin{proposition}\label{prop:uniqaptan}
Let $E \subset \bH^n$ be $\H^{k_m}$-measurable with $\H^{k_m}(E)<\infty$, and let $A$ be the set of points of $E$ for which there is an approximate tangent parabolid of dimension $k$ and metric dimension $k_m$. Then the following holds:
\begin{enumerate}[label=(\alph*)]
\item $A$ is $\H^{k_m}$-measurable;
\item $E$ has an unique approximate tangent paraboloid $V_p$ at $\H^{k_m}$-a.e. $p \in A$;
\item the mapping $A \ni p \mapsto V_p \in \cG(\bH^n,k)$ is measurable.
\end{enumerate}
\end{proposition}

The proof of this result goes along the same lines of \cite[Proposition 3.9]{MSSC10}, so we refer the reader to that paper for more details.

We are now ready to prove our main result. For this, we follow closely the strategy of \cite[Theorem 3.15]{MSSC10} and only point out the main differences in our case.

\begin{proof}[Proof of \Cref{thm.1.1}]
Since $\cH^{k_m}(E)<\infty$, by a standard density estimate (\Cref{lemma:densities})
\[
\Theta^{\ast k_m}(E,p)\le 5^{k_m} \qquad \text{for } \cH^{k_m}\text{-a.e. } p\in E,
\]
so we can assume without loss of generality that
\beq \label{uppdensitybd}
\cH^{k_m}(E\cap B(p,r))\le 
7^{k_m}r^{k_m}\qquad \text{for all } p\in E \text{ and } r>0.
\eeq
Using the positive lower density and \eqref{eq.paraboloidassumption}, we have that for $\H^{k_m}$-a.e. $p \in E$ we can find $\ell(p)>0$, $0<r(p)\le 1$ and $V_p = \mathrm{apPar}_\bH^{k_m}(E,p)$ such that
\beq \label{eq.3.17}
\cH^{k_m}\left(E \cap B(p,r) \right) > \ell(p) r^{k_m} \qquad \text{for all } 0 < r < r(p),
\eeq
and, for some $\lambda:=\lambda_p > 0$, also that
\[
\cH^{k_m} \left(E\cap B(p,r) \sm Q_\alpha(p,V_p,\lambda)\right)\le \eps r^{k_m}
\]
where $\eps < \ell(p)/(4^{k_m}+1)$. Moreover, since $B(p,r)\cap Q_\alpha(p,V_p,\lambda)\subset C^r_\alpha(p)$, we have
\beq\label{eps-cyluppbd}
\cH^{k_m} \left(E\cap B(p,r) \sm C^r_\alpha(p)\right)\le \eps r^{k_m}.
\eeq
Now consider for every $i \geq 1$ the set
\[
E_i := \left\{ p \in E \ : \ \min\{r(p),\ell(p)\} > \frac{1}{i} \right\},
\]
and define $E^\ast:=\cup_{i\ge 1} E_i$. Then clearly $\H^{k_m}(E\sm E^*)=0$, which means that it is enough to prove the result for the set $E_i$ for every $i\ge 1$.

Now recall that, by \Cref{lemma.2.32}, for any $p \in E^\ast$ we can find $2n+1-k$ unit vectors $\nu_h(p)$ in the horizontal bundle $\mathrm{H}\bH_p^n$ transversal to $V_p$ and such that
\[
T_p := \exp\left( \operatorname{span} \{\nu_1(p),\ldots,\nu_{2n+1-k}(p)\} \right)
\] 
is a horizontal subgroup of $\bH^n$ satisfying $\bH^n = V_p \cdot T_p$. Moreover, using the Lipschitz continuity (see \Cref{lemma.2.32}) and \Cref{prop:uniqaptan}, we obtain
\[
E^\ast = \bigcup_{j \ge 1} F_j,
\]
with $\cH^{k_m}(F_j) < \infty$ and $\nu_h \, \big|_{F_j}$ is $\H^{k_m}$-measurable. As a consequence, the map
\[
\nu_h : E^\ast \to \mathrm{H}\bH^n
\]
is a measurable sections of $\mathrm{H}\bH^n$ for every $1 \leq h \leq 2n+1-k$.

Now, following \cite{MSSC10}, we define for appropriate indexes and for $p \in E^\ast$ the function
\[
\rho_{i,h,j}(p) := \sup \left\{ \frac{|\langle \nu_h(p), \pi(p^{-1}q)\rangle|}{d(p,q)^{1+\alpha}} \ : \  q \in E_i, \ 0 < d(p,q) < \frac{1}{j}\right\},
\]
where $\pi : \bH^n \to \ch_1$ is the projection onto the first layer given by
\[
\pi(q) = \sum_{j = 1}^{n} (q_j X_j + q_{n+j} Y_j).
\]
For every couple of points $p,q\in E_i$, applying \Cref{lemma:keyHn} with $r:=d(p,q)$ we obtain that the map $p\mapsto V_p$ is $\alpha$-H\"{o}lder continuous when restricted to $E_i$, namely
\[
\rho(V_p,V_q)\le Cd(p,q)^\alpha\quad \text{ for every } p,q\in E_i,
\]
where $C:=C(n,\ell(p),\lambda)$ is the constant given in \Cref{lemma:keyHn}. Notice that this is possible because we use \eqref{uppdensitybd} to verify the assumption \eqref{eq:upperhypHn}. We now claim that, up to taking a larger aperture $\lambda'>\lambda$, we have that
\[
E_i \sm Q_\alpha(p,V_p,\lambda')=\varnothing.
\]
We argue by contradiction. Let $p,q\in E_i$ be such that the assumptions of \Cref{lemma.2.29} hold; then we have the inclusion
\[
B\left(q,\frac{r}{2}\right)\cap C^{r/2}_\alpha(q) \subset B(p,2r)\setminus Q_\alpha(p,V_p,\lambda),
\]
under the additional assumptions $q \notin Q_\alpha(p,V_p,\lambda')$ and $\lambda'$ satisfying \eqref{lambdapvalue}. Hence, using the estimates \eqref{eq.3.17} and \eqref{eps-cyluppbd}, we deduce that
\[ \begin{aligned}
(\ell(p)-\eps)(r/2)^{k_m}&\le\cH^{k_m}\left(E_i\cap B\left(q,r/2\right)\cap C^{r/2}_\alpha(q)\right)
\\ & \le \cH^{k_m}\left(E\cap B\left(p,2r\right)\setminus Q_\alpha(p,V_p,\lambda)\right)\le 2^{k_m}\eps^{k_m},
\end{aligned}\]
which is a contradiction since $\eps<\ell(p)/(4^{k_m}+1)$. This means that the claim holds and, by \Cref{lemma.2.32}, it follows that
\beq\label{Key-inclusion}
E_i \sm Q_\alpha(p,\N(\nu_h(p)),\lambda')=\varnothing \qquad \text{for all } 1\le h\le 2n+1-k.
\eeq
As a consequence of \eqref{Key-inclusion}, for every $i \ge 1$ and $1\le h \le 2n+1-k$ we have
\[
\lim_{j\to \infty}\rho_{i,h,j}(p) = 0.
\]
If we now apply Lusin theorem to each $\nu_h$ and Egoroff theorem to the corresponding sequence $\left(\rho_{i,h,j}\right)_{j \in \N}$, we obtain a decomposition of each $E_i$ given by
\[
E_i = E_{i,0} \bigcup \left(\bigcup_{\beta \geq 1} K_{i,\beta}\right),
\]
where $E_{i,0}$ is $\H^{k_m}$-negligible, $K_{i,\beta}$ is compact, $\nu_h \, \big|_{K_{i,\beta}}$ is continuous and $\rho_{i,h,j}$ goes to zero uniformly in $K_{i,\beta}$ with respect to $j$. By \Cref{lemma.2.32.holder}, we also find that
\[
\nu_h \, \big|_{K_{i,\beta}}
\]
is $\alpha$-H\"{o}lder continuous and, by applying the Whitney theorem (see \Cref{Lemma:Whitney}) on $K_{i,\beta}$, we obtain the extended functions $f_{i,\beta,h} \in \mathbf{C}_H^{1,\alpha}(\bH^n)$ such that 
\[
f_{i,\beta,h} \, \big|_{K_{i,\beta}} = 0, \quad \nabla_H f_{i,\beta,h} \, \big|_{K_{i,\beta}} = \nu_h \quad \text{and} \quad |\nabla_H f_{i,\beta,h}| \neq 0 \text{ on } K_{i,\beta}.
\] 
Finally, notice that the set
\[
S_{i,\beta,h} := \{ p \in \bH^n \ : \ f_{i,\beta,h}(p) = 0,\ |\nabla_H f_{i,\beta,h}| \neq 0 \}
\]
is a $1$-codimensional $(\mathbf{C}_H^{1,\alpha},\bH)$-regular surface containing $K_{i,\beta}$; thus, we can consider the following intersection:
\[
S_{i,\beta} := \bigcap_{h = 1}^{2n+1-k} S_{i,\beta,h}.
\]
By \Cref{prpps23} we have that $S_{i,\beta}$ is a $k$-codimensional $(\mathbf{C}_H^{1,\alpha},\bH)$-regular surface that contains the set $K_{i,\beta}$. Moreover, we have
\[
E \subset E_0 \cup (\cup_{i \geq 1} \cup_{\beta \geq 1} S_{i,\beta}),
\]
with $E_0=(E\sm E^*)\cup(\cup_{i \ge 1} E_{i,0})$ and $\cH^{k_m}(E_0)=0$. Hence $E$ is $C^{1,\alpha}$-rectifiable.
\end{proof}

\section{Extension to Carnot groups} \label{sec:carnot}

Let $\G=H^1\oplus\cdots\oplus H^\iota$ be a Carnot group of step $\iota$, dimension $q$, and homogeneous dimension given by
\[
Q = \sum_{j = 1}^\iota j \cdot \dim (H^j).
\]
In analogy with the Heisenberg case, we say that a homogeneous subgroup $T \subset \G$ of codimension $k \leq \dim(H^1)$ is {\em vertical} if
\[
T = T_H \oplus H^2 \oplus \cdots \oplus H^\iota,
\]
where $T_H \subset H^1$ has dimension $\dim(H^1) - k$. As a consequence, the definition of the Grassmannian (\Cref{def.grass}) can be immediately extended.

\begin{remark}
This extension leads to an immediate issue. Indeed, if $S$ is a vertical subgroup of codimension at least $2$, then there is no guarantee that a {\bfseries horizontal} complement exists (so $S$ may not belong to the Grassmannian).
\end{remark} 

\begin{remark}\label{rmk.ext}
In the special case of codimension one, a horizontal complement always exists since it is generated by a single element. Thus, if we replace
\[
\text{$k_m$ with $Q-1$ and $k$ with $q-1$},
\]
then every result in this paper, except for \Cref{lemma.2.32.holder}, can be proved following the same exact argument. Luckily, the exception is trivial since the distance between horizontal complements is given by the distance between the two generating elements, which can be chosen using the orthogonality condition.
\end{remark}

\subsection*{Acknowledgements.} 
The authors would like to thank Giacomo Del Nin for fruitful discussions on the subject and  Davide Vittone for pointing out the one-codimensional case for Carnot groups and also for valuable suggestions regarding the original draft.

\addtocontents{toc}{\protect\setcounter{tocdepth}{2}}

\bibliography{bibliography}

\end{document}